\documentclass[12pt]{article}

\usepackage{amsmath,amsthm}
\usepackage[hidelinks]{hyperref}
\usepackage{bussproofs}
\usepackage[tableaux]{prooftrees}
\usepackage{natbib}

\newtheorem{thm}{Theorem}
\newtheorem{prop}[thm]{Proposition}
\newtheorem{cor}[thm]{Corollary}
\theoremstyle{definition}
\newtheorem{defn}[thm]{Definition}

\def\nec#1#2{[\,#1\,]#2}
\def\poss#1#2{\langle#1\rangle#2}
\def\Prop#1{[\![#1]\!]}
\def\pr#1#2#3{\ensuremath{#1 \mathrel{#3} #2}}
\def\pf#1#2{\ensuremath{#1 : #2}}
\def\R#1#2#3{#2 \mathrel{R_{#1}} #3}
\def\Rf#1#2{R_{#1}[#2]}
\def\VC{{\textbf{VC}}}
\def\Vc{{\textbf{Vc}}}
\def\Ck{{\textbf{Ck}}}
\def\CK{{\textbf{CK}}}

\let\phi\varphi
\def\lif{\mathbin{\supset}}

\providecommand{\doi}[1]{doi: \href{http://dx.doi.org/#1}{#1}}

\DeclareSymbolFont{symbolsC}{U}{pxsyc}{m}{n}

\DeclareMathSymbol{\cif}{\mathrel}{symbolsC}{128}
\DeclareMathSymbol{\mif}{\mathrel}{symbolsC}{132}

\let\sat\vDash\let\nsat\nvDash

\usepackage{fancyhdr}
\fancypagestyle{firststyle}
{
   \fancyhf{}
   \fancyhead[L]{\large\textsc{Australasian Journal of Logic}}
   \fancyfoot[RE,LO]{\small \textmd{Australasian Journal of Logic (15:3) 2018, Article no. 1 }}
}
\author{Richard Zach\thanks{University of Calgary, Department of Philosophy, 2500 University Dr NW, Calgary AB T2N 1N4, Canada, \texttt{rzach\@ucalgary.ca}}} \title{Non-Analytic Tableaux for Chellas's Conditional Logic
  \CK{} and Lewis's Logic of Counterfactuals~\VC}
\date{}

\begin{document}
\bibliographystyle{plainnat}

\maketitle

\begin{abstract}
  Priest has provided a simple tableau calculus for Chellas's
  conditional logic~\Ck. We provide rules which, when added to
  Priest's system, result in tableau calculi for Chellas's~\CK{} and
  Lewis's~\VC. Completeness of these tableaux, however, relies on the
  cut rule.
\end{abstract}

\thispagestyle{firststyle}

\section{Introduction}

\citet{Chellas1975} presented a conditional logic \CK{} that is sound
and complete for a relational semantics. The system uses an
intensional conditional~$\Rightarrow$ which plays the same role as
the counterfactual conditional~$\cif$ of \cite{Lewis1973}.  Chellas
realized that the conditional $\phi \cif \psi$ can be seen as a
necessity operator in which the accessibility relation is indexed by
the proposition expressed by the antecedent. He suggested the
notation $\nec\phi\psi$. Lewis's dual might conditional
$\phi \mif \psi$ would then correspond to an indexed possibility
operator~$\poss\phi\psi$.  \cite{Segerberg1989} extended Chellas's
semantics and showed that Lewis's logic of counterfactuals \VC{} is an
axiomatic extension of \CK{} complete relative to the class of
``Segerberg models'' that satisfy certain restrictions.

In note~14 of his paper, \citet{Chellas1975} suggests an alternative
semantics in which the accessibility relation is not indexed by
propositions but by formulas themselves. While in \CK{} and its
extensions $\nec\phi\theta$ is equivalent to $\nec\psi\theta$ whenever
$\phi$ and $\psi$ are equivalent (always true at the same worlds),
this is not so when the accessibility relation is indexed by the
formulas~$\phi$ and $\psi$. The result is a logic~\Ck{} with a simpler
semantics. It is characterized by the rules and axioms of \CK{}
without the rule~RCEA. \citet{Priest2008} discussed Chellas's \Ck{} in
more depth and provided a sound and cut-free complete tableau system
for it.

We give additional branch extension rules for Priest's system which
result in sound and complete tableau systems for \CK{} and~\VC. These
systems are, however, non-analytic in that the cut rule is
included. The completeness proof proceeds by showing that tableau
provability is closed under the rules and axioms of \VC{} as given by
Chellas and Segerberg.

These are not the first tableau systems for
Lewis's~\VC. \citet{Gent1992}, based on work of \cite{Swart1983}, has
given a tableau system for~\VC. More recently, \citet{Negri2016} have
offered a cut-free complete sequent calculus for~\VC. These approaches
are all based on Lewis's semantics based on relative proximity of
worlds and incorporate the ordering relation between worlds into the
syntax. The present approach is perhaps more perspicuous and holds
promise for other logics based on the Chellas-Segerberg approach, such
as those studied by \citet{Unterhuber2013} and
\citet{Unterhuber2014}. A major drawback of our proposal is of course
the presence of the cut rule.  Semantic proofs of cut-free
completeness face significant challenges, which we discuss.

\section{Syntax and Semantics}

The syntax of Chellas's \CK{} is that of propositional logic with the
addition of an indexed necessity operator~$\nec\phi{}$:

\begin{defn}
  Formulas are defined inductively:
  \begin{enumerate}
  \item Every propositional variable $p$ is a formula.
  \item $\bot$ is a formula.
  \item If $\phi$ is a formula, so is $\lnot \phi$.
  \item If $\phi$ and $\psi$ are formulas, so are
    \[
    (\phi \lif \psi), (\phi \land \psi), (\phi \lor \psi),
    \nec\phi\psi.
    \]
  \end{enumerate}
\end{defn}

We can define $\top$ as $\lnot \bot$, $\poss\phi\psi$ as
$\lnot\nec\phi{\lnot\psi}$, and $\phi \equiv \psi$ as $(\phi \lif
\psi) \land (\psi \lif \phi)$. We can also define $\Box \phi$ as
$\nec{\lnot\phi}\bot$ and $\Diamond\phi$ as $\poss\phi\top$ (or,
$\lnot\nec\phi\bot$). If Lewis's notation is preferred, then read
$\phi\cif\psi$ as $\nec\phi\psi$ and $\phi\mif\psi$ as
$\poss\phi\psi$, i.e., $\lnot\nec\phi{\lnot\psi}$.  Segerberg
preferred the notations $\phi \mathop{\sqsupset} \psi$ and $\phi
\mathop{>} \psi$ to Chellas's $\nec\phi\psi$ and $\poss\phi\psi$.

Chellas also provided a relational semantics for which his logic \CK{}
is sound and complete. However, that semantics does not cover
Lewis's~\VC{}.  Segerberg provided the required generalization: a
semantics for which the axiomatic system for \VC{} is sound and
complete.

\begin{defn}
  A Segerberg model $M = \langle U, P, R, V\rangle$ consists of a set
  of worlds~$U \neq \emptyset$, a set of propositions~$P \subseteq
  \wp(U)$, a propositionally indexed accessibility relation $R\colon P \to
  \wp(U \times U)$, and a variable assignment $V\colon \mathit{Var} \to P$.

  We write $R_S$ for $R(S)$ and $\Rf S x$ for $\{ y \mid \R S x y\}$.
  
  The set of propositions must contain~$\emptyset$, be closed under
  intersection, complement, and necessitation, i.e.,
  \begin{enumerate}
  \item $\emptyset \in P$, $U \in P$.
  \item If $S \in P$, then $U \setminus S \in P$.
  \item If $S$, $T \in P$, then $S \cap T \in P$ and $S \cup T \in P$.
  \item If $S$, $T \in P$, then $\{x \in U \mid \Rf S x \subseteq T\} \in P$.
  \end{enumerate}
\end{defn}

\begin{defn}
  Truth of a formula $\phi$ at a world~$x$ in $M$, $M, x \sat \phi$, 
  is defined by:
  \begin{enumerate}
  \item $M, x \nsat \bot$ always.
  \item $M, x \sat p$ iff $x \in V(p)$.
  \item $M, x \sat \phi \land \psi$ iff $M, x \sat \phi$ and $M, x \sat \psi$.
  \item $M, x \sat \phi \lor \psi$ iff $M, x \sat \phi$ or $M, x \sat \psi$.
  \item $M, x \sat \phi \lif \psi$ iff $M, x \nsat \phi$ or $M, x \sat \psi$.
  \item $M, x \sat \nec\phi\psi$ iff for all~$y$ such that $\R \phi x
    y$, $M, y \sat \psi$.
  \item $M, x \sat \poss\phi\psi$ iff for some~$y$ such that $\R \phi x
    y$, $M, y \sat \psi$.
  \end{enumerate}
  We write $\Prop\phi$ for $\{x \mid M, x \sat \phi\}$ and $R_\phi$ for
  $R_{\Prop\phi}$.
  
  We say that $M$ satisfies $\phi$ if $M, x \sat \phi$ for some $x \in
  U$; $M$ satisfies $\Gamma$ if for some $x \in U$, $M, x \sat \phi$
  for all $\phi \in \Gamma$; $\Gamma$ is satisfiable if some $M$
  satisfies~$\Gamma$.
\end{defn}

\begin{defn}
  We say that $\Gamma$ \emph{\CK-entails} $\phi$, $\Gamma \sat_\CK
  \phi$, iff for every Segerberg model~$M$ and world~$x$ such that $M,
  x \sat \psi$ for every $\psi \in \Gamma$, $M, x \sat \phi$.
\end{defn}

\section{Tableaux for \Ck{} and \CK}

\begin{defn}
  \emph{Prefixed formulas} are expressions of the form $\pf{i}{\phi}$
  or $\pr{i}{j}{\phi}$.
\end{defn}

Intuitively, $\pf{i}{\phi}$ means ``$\phi$ is true at~$i$'' and $\pr i j
\phi$ means ``$j$ is $\phi$-accessible from~$i$.''

\begin{table}
  \begin{center}
  \begin{tabular}{cc}
    \hline\\
      \AxiomC{\pf i {\phi \land \psi}}
      \RightLabel{$\land$}
      \UnaryInfC{\pf i {\phi}}
      \noLine
      \UnaryInfC{\pf i {\psi}}
      \DisplayProof
    &
      \AxiomC{\pf i {\lnot(\phi \land \psi)}}
      \RightLabel{$\lnot\land$}
      \UnaryInfC{\pf i {\lnot\phi} \qquad \pf i {\lnot\psi}}
      \DisplayProof
\\[5ex]
      \AxiomC{\pf i {\phi \lor \psi}}
      \RightLabel{$\lor$}
      \UnaryInfC{\pf i {\phi} \qquad \pf i {\psi}}
      \DisplayProof
      &
      \AxiomC{\pf i {\lnot(\phi \lor \psi)}}
      \RightLabel{$\lnot\lor$}
      \UnaryInfC{\pf i {\lnot\phi}}
      \noLine
      \UnaryInfC{\pf i {\lnot\psi}}
      \DisplayProof
\\[5ex]
      \AxiomC{\pf i {\phi \lif \psi}}
      \RightLabel{$\lif$}
      \UnaryInfC{\pf i {\lnot\phi} \qquad \pf i {\psi}}
      \DisplayProof
      &
      \AxiomC{\pf i {\lnot(\phi \lif \psi)}}
      \RightLabel{$\lnot\lif$}
      \UnaryInfC{\pf i {\phi}}
      \noLine
      \UnaryInfC{\pf i {\lnot\psi}}
      \DisplayProof
      \\[5ex]
      \multicolumn{2}{c}{\AxiomC{\pf i {\lnot\lnot\phi}}
      \RightLabel{$\lnot$}
      \UnaryInfC{\pf i {\phi}}\DisplayProof}
      \\[3ex]
      \AxiomC{\pf i {\nec\phi\psi}}
      \noLine
      \UnaryInfC{\pr i j \phi}
      \RightLabel{$\Box$}
      \UnaryInfC{\pf j \psi}
      \DisplayProof
      &
      \AxiomC{\pf i {\lnot\nec\phi\psi}}
      \RightLabel{$\lnot\Box$}
      \UnaryInfC{\pf j {\lnot\psi}}
      \noLine
      \UnaryInfC{\pr i j \phi}
      \DisplayProof
      \\[5ex]
      \AxiomC{\pf i {\poss\phi\psi}}
      \RightLabel{$\Diamond$}
      \UnaryInfC{\pf j \psi}
      \noLine
      \UnaryInfC{\pr i j \phi}
      \DisplayProof
      &
      \AxiomC{\pf i {\lnot\poss\phi\psi}}
      \noLine
      \UnaryInfC{\pr i j \phi}
      \RightLabel{$\lnot\Diamond$}
      \UnaryInfC{\pf j {\lnot\psi}}
      \DisplayProof\\[5ex]
      \multicolumn{2}{c}{In the $\lnot\Box$ and $\Diamond$ rules, $j$
        must be new to the branch.}\\ \hline
  \end{tabular}\end{center}
  
  \caption{Branch extension rules for Priest's tableau system for
    Chellas's \Ck}\label{rules-CK}
\end{table}

\begin{defn}
  Let $\Gamma$ be a set of prefixed formulas.  A \emph{tableau} is a
  downward-branching tree labelled by prefixed formulas such that
  every prefixed formula in the tree is either in~$\Gamma$ or is a
  conclusion of a branch extension rule. If a formula is one
  conclusion of a branching rule, its siblings in the tree must be
  labelled with the other conclusions.

  A branch of a tableau is \emph{closed} if it
  contains both $\pf i \phi$ and $\pf i {\lnot\phi}$ for some $i$
  and~$\phi$, or it contains $\pf i \bot$.

  We write $\Gamma \vdash \phi$ if $\{\pf 1 \psi \mid \psi \in \Gamma\}
  \cup \{\pf 1 \lnot\phi\}$ has a closed tableau.
\end{defn}

We say a tableau is a \Ck-tableau if it only uses the rules of
\hyperref[rules-CK]{Table~\ref*{rules-CK}}, and a \Vc-tableau
if it also used the rules of
\hyperref[rules-VC]{Table~\ref*{rules-VC}}. A \CK- or
\VC-tableau is a \Ck- or \Vc-tableau that in
addition uses the rules
\[
\AxiomC{}
\RightLabel{cut}
\UnaryInfC{$\pf i \phi \qquad \pf i {\lnot \phi}$}
\DisplayProof\qquad
\AxiomC{\pr i j \phi}
\RightLabel{ea}
\UnaryInfC{$\begin{array}{l} \pf k {\lnot \phi}\\ \pf k \psi\end{array}
    \qquad
    \begin{array}{l} \pf k \phi\\ \pf k {\lnot \psi}\end{array}
    \qquad \pr i j \psi$}
\DisplayProof
\]
In the cut rule, the index~$i$ must already occur on the branch. 
In the ea rule, the index $k$ must be new to the branch.

Both cut and ea are non-analytic in the sense that the conclusion
formulas are not subformulas of the premises. This is an essential
property of cut. The rule ea could, however, be replaced by the following
analytic rule:
\[
\AxiomC{\pf i {\nec\psi\theta}}
\noLine
\UnaryInfC{\pr i j \phi}
\RightLabel{ea$'$}
\UnaryInfC{$\begin{array}{l} \pf k {\lnot \phi}\\ \pf k \psi\end{array}
    \qquad
    \begin{array}{l} \pf k \phi\\ \pf k {\lnot \psi}\end{array}
    \qquad \pf j \theta$}
\DisplayProof
\]
This rule could also replace the $\Box$ rule in tableau systems for
\CK{} and~\VC.

We write $\Gamma \vdash_\Ck \phi$, etc., to indicate that there is a
closed \Ck-tableau that shows $\Gamma \vdash \phi$.

\begin{prop}\label{axioms-Ck}
  There are $\Ck$-proofs of the following:
  \begin{align}
    \nec\phi{(\psi \land \theta)} &
    \vdash \nec\phi\psi \land \nec\phi\theta \tag{CM} \\
    \nec\phi\psi \land \nec\phi\theta &
    \vdash \nec\phi{(\psi \land \theta)} \tag{CC} \\
    & \vdash \nec\phi\top \tag{CN}
  \end{align}
\end{prop}

\begin{proof}
  See \hyperref[appendix]{Appendix}.
\end{proof}

CM and CC together are Segerberg's M1, and CN is Segerberg's M2.

\begin{prop}\label{RCEC}
  If $\psi \vdash_\Ck \theta$, then $\nec\phi\psi \vdash_\Ck \nec\phi\theta$
\end{prop}

\begin{proof}
  Since $\psi \vdash_\Ck \theta$, there is a closed \Ck-tableau with
  assumptions $\pf 1 \psi, \pf 1 {\lnot\theta}$. If we raise every
  index in it by $1$ it remains a closed tableau. Now consider:
\begin{center}
  \begin{tableau}{}
    [\pf 1 \nec\phi\psi, just = Ass
      [\pf 1 {\lnot\nec\phi\theta}, just = Ass
        [\pr 1 2 \phi, just = $\lnot\Box$:!u
          [\pf 2 \lnot\theta, just = $\lnot\Box$:!uu
            [\pf 2 \psi, just = $\Box$:!uuuu
              [\vdots,close]
            ]
          ]
        ]
      ]
    ]
  \end{tableau}
\end{center}
where the part indicated by $\vdots$ is the above closed tableau for
$\psi \vdash \theta$ with indices raised by~1 but without its
assumptions (which however appear on lines 4 and~5). Since that
tableau closes, the resulting tableau closes.
\end{proof}

This establishes the derivability of Chellas's rule RCEC, aka Segerberg's~EC:
\[
\AxiomC{$\psi \equiv \theta$}
\RightLabel{RCEC}
\UnaryInfC{$\nec\phi\psi \equiv \nec\phi\theta$}
\DisplayProof
\]

\begin{prop}\label{RCEA}
  If $\phi \vdash_\CK \psi$ and $\psi \vdash_\CK \phi$, then
  $\nec\phi\theta \vdash_\CK \nec\psi\theta$.
\end{prop}

\begin{proof}
  Consider the tableau
  \begin{center}
    \begin{tableau}{}
      [\pf 1 {\nec\phi\theta}, just=Ass
        [\pf 1 {\lnot\nec\psi\theta}, just=Ass
          [\pf 2 {\lnot\theta}, just=$\lnot\Box$:!u
            [\pr 1 2 \psi, just=$\lnot\Box$:!uu
              [\pf 3 \lnot\psi, just =ea:!u
                [\pf 3 \phi, just=ea:!uu
                  [\vdots, close]
                ]
              ]
              [\pf 3 \psi, just =ea:!u
                [\pf 3 \lnot\phi, just=ea:!uu
                  [\vdots, close]
                ]
              ]
              [\pr 1 2 \phi, just = ea:!u
                [\pf 2 \theta, just=$\Box$:{!uuuuu,!u},
                  close={:!uuu,!c}, move by=2]
              ]
            ]
          ]
        ]
      ]
    \end{tableau}
    \end{center}
  The parts indicated by $\vdots$ are the closed tableaux for $\phi
  \vdash \psi$ and $\psi \vdash \phi$, respectively, with all indices
  raised by~$2$ and the assumption removed.
\end{proof}

This establishes the derivability of Chellas's rule RCEA, aka Segerberg's~EA:
\[
\AxiomC{$\phi \equiv \psi$}
\RightLabel{RCEA}
\UnaryInfC{$\nec\phi\theta \equiv \nec\psi\theta$}
\DisplayProof
\]

\begin{prop}\label{MP}
  In the presence of cut, if $\Gamma \vdash \phi$ and $\phi \vdash
  \psi$, then $\Gamma \vdash \psi$.
\end{prop}

\begin{proof}
  If $\Gamma = \{\theta_1, \dots, \theta_n\}$, $\Gamma \vdash \phi$
  means there is a closed tableau for
  \[
  \pf 1 {\theta_1}, \dots, \pf 1 {\theta_n}, \pf 1 {\lnot\phi}.
  \]
  Using cut, we can construct a closed tableau for
  \[
  \pf 1 {\theta_1}, \dots, \pf 1 {\theta_n}, \pf 1 {\lnot\psi}
  \]
  as follows:
  \begin{center}
  \begin{tableau}{not line numbering}
    [\pf 1 {\theta_1}, just=Ass
      [\vdots
        [\pf 1 {\theta_n}, just=Ass
          [\pf 1 {\lnot\psi}, just = Ass
            [\pf 1 \phi, just = cut
              [\vdots, close]
            ]
            [\pf 1 {\lnot\phi}, just = cut
              [\vdots, close]
            ]
          ]
        ]
      ]
    ]
  \end{tableau}
  \end{center}
  The sub-tableau on the left is the closed tableau for $\phi \vdash
  \psi$, since that branch contains both $\pf 1 \phi$ and $\pf 1
      {\lnot \psi}$.  The one on the right is the closed tableau for
      $\Gamma \vdash \phi$, since that branch contains all of $\pf 1
      {\theta_1}$, \dots, $\pf 1 {\theta_n}$, $\pf 1 {\lnot\phi}$.
\end{proof}

\begin{thm}\label{complete-CK}
  \begin{enumerate}
  \item If there is a derivation using RCEC and modus ponens using
    tautologies and axioms CM, CC, CN of $\phi$ from $\Gamma$, then
    there is a \Ck-tableau with cut that shows $\Gamma \vdash_{\Ck +
      \mathrm{cut}} \phi$.
  \item If there is a derivation using RCEA, RCEC, and modus ponens
    using tautologies and axioms CM, CC, CN of $\phi$ from $\Gamma$, then
    there is a \CK-tableau that shows $\Gamma \vdash_\CK \phi$.
  \end{enumerate}
\end{thm}

\begin{proof}
  \begin{enumerate}
    \item Propositional tautologies have closed tableaux since
      the rules for the ordinary propositional connectives are
      complete. The axioms CM, CC, CN have closed tableaux by
      \hyperref[axioms-Ck]{Proposition~\ref*{axioms-Ck}}. Closure of
      $\vdash_\Ck$ under modus ponens is established by
      \hyperref[MP]{Proposition~\ref*{MP}} (this requires
      cut). Closure under RCEC is established by
      \hyperref[RCEC]{Proposition~\ref*{RCEC}}.
    \item Follows from (1) and
      \hyperref[RCEA]{Proposition~\ref*{RCEA}}.
  \end{enumerate}
\end{proof}

\begin{cor}
  If $\Gamma \sat_\CK \phi$ then $\Gamma \vdash_\CK \phi$.
\end{cor}

\begin{proof}
  \citet{Segerberg1989} showed strong completeness of the system RCEA,
  RCEC, CM, CC, CN of \CK{} for Segerberg models. The result follows
  by \hyperref[complete-CK]{Theorem~\ref*{complete-CK}(2)}.
\end{proof}

We defer proofs of soundness to
\hyperref[soundness-sub-VC]{Corollary~\ref*{soundness-sub-VC}}.

\section{The systems \Vc{} and \VC}

\citet{Segerberg1989} has given an axiomatization of Lewis's logic
\VC{} of counterfactuals in the framework set out by Chellas, and
proven completeness for Segerberg models that satisfy a number of
conditions. The axioms that have to be added to \CK, and the
corresponding conditions on the models, are given in
\hyperref[axioms-conditions]{Table~\ref*{axioms-conditions}}.  The
corresponding tableau system is the tableau system for~\CK{}
plus the additional rules given in
\hyperref[rules-VC]{Table~\ref*{rules-VC}}.

We call a Segerberg model which satisfies the conditions of
\hyperref[axioms-conditions]{Table~\ref*{axioms-conditions}} a
\emph{\VC-model}.

\begin{table}
  \begin{center}
  \begin{tabular}{l|l|l}
    & Axiom & Condition\\
    \hline\hline
    S1 & $\nec\phi\phi$ & $\Rf S x \subseteq S$\\
    S2 & $\poss\phi\psi \lif \poss\psi\top$ &
    $\Rf S x \cap T \neq \emptyset \Rightarrow \Rf T x \neq \emptyset$\\
    S3 & $\phi \lif \nec\top\phi$ &
    $\Rf U x \subseteq \{x\}$ \\
    S4 & $\phi \lif \poss\top\phi$ &
    $x \in \Rf U x$\\
    S5 & $\nec{\phi \land \psi}\theta \lif \nec\phi{(\psi \lif \theta)}$ &
    $\Rf S x \cap T \subseteq \Rf {S \cap T}{x}$\\
    S6 & $\poss\phi\psi \lif (\nec\phi{(\psi \lif \theta)} \lif {}$ &
    $\Rf S x \cap T \neq \emptyset \Rightarrow {}$\\
    & $\quad\nec{\phi \land \psi}\theta$ & $\quad  \Rf{S \cap T}{x} \subseteq \Rf S x \cap T$
    \\\hline
\end{tabular}\end{center}
    \caption{Segerberg's axioms for \VC{} and corresponding conditions}
    \label{axioms-conditions}
\end{table}

\begin{table}
  \begin{center}
  \begin{tabular}{cc}
    \hline\\
      \AxiomC{\pr i j \phi}
      \RightLabel{R1}
      \UnaryInfC{\pf j \phi}
      \DisplayProof
    &
      \AxiomC{\quad}
      \RightLabel{R4}
      \UnaryInfC{\pr i i \top}
      \DisplayProof
\\[4ex]
      \AxiomC{\pf j \psi}
      \noLine
      \UnaryInfC{\pr i j \phi}
      \RightLabel{R2}
      \UnaryInfC{\pr i k \psi}
      \DisplayProof
      &
      \AxiomC{\pf j \psi}
      \noLine
      \UnaryInfC{\pr i j \phi}
      \RightLabel{R5}
      \UnaryInfC{\pr i j {\phi \land \psi}}
      \DisplayProof
      \\[5ex]
      \AxiomC{\pf i \phi}
      \noLine
      \UnaryInfC{\pf j {\lnot\phi}}
      \noLine
      \UnaryInfC{\pr i j \top}
      \RightLabel{R3}
      \UnaryInfC{\pf j \phi}
      \DisplayProof
      &
      \AxiomC{\pf j \psi}
      \noLine
      \UnaryInfC{\pr i j \phi}
      \noLine
      \UnaryInfC{\pr i k {\phi \land \psi}}
      \RightLabel{R6}
      \UnaryInfC{\pf k \psi}
      \noLine
      \UnaryInfC{\pr i k \phi}
      \DisplayProof
      \\[7ex]
      \multicolumn{2}{p{.6\textwidth}}{In R2, $k$
        must be new to the branch. In R4, $i$ must occur on the branch.}
      \\ \hline
  \end{tabular}\end{center}
  
  \caption{Branch extension rules for \VC}\label{rules-VC}
\end{table}

We show that our tableau system for \VC{} is sound and complete for
\VC-models. For soundness, we have to extend the definition of
satisfaction to prefixed formulas.

\begin{defn}
  Suppose $M$ is a Segerberg model, $\Gamma$ is a set of prefixed
  formulas, and $f\colon I \to U$ where $I$ is the set of indices
  occurring in~$\Gamma$. We say $M, f$ \emph{satisfies}~$\Gamma$ iff
  \begin{enumerate}
  \item If $\pf{i}{\phi} \in \Gamma$, then $M, f(i) \sat \phi$.
  \item If $\pr{i}{j}{\phi} \in \Gamma$, then $\R{\phi}{f(i)}{f(j)}$.
  \end{enumerate}
\end{defn}

\begin{thm}\label{soundness-VC}
  Tableaux for~\VC{} are sound for \VC-models, i.e., any
  set of formulas with a closed tableau is not satisfiable.
\end{thm}

\begin{proof}
  We show that if a satisfiable branch is extended by an application
  of a rule, at least one resulting branch is satisfiable. Thus, every
  tableau starting from a satisfiable set of assumptions~$\Delta$
  contains at least one satisfiable branch and thus cannot be closed.

  If $\Delta$ is satisfiable, then for some $M$ and $x \in U$, $M, x
  \sat \phi$ for all~$\phi \in \Delta$. Let $f\colon \{1\} \to U$ be
  given by $f(1) = x$. Then $M, f$ satisfies the assumptions
  $\pf{1}{\theta_1}, \dots, \pf{1}{\theta_n}$ of the tableau where
  $\Delta = \{\theta_1, \dots, \theta_n\}$.

  Now let $\Gamma$ be the set of prefixed formulas on a branch satisfied
  by $M, f$. The cases for the rules given in
  \hyperref[rules-CK]{Table~\ref*{rules-CK}} are routine; we carry out
  the cases for $\nec\phi\psi$ and $\lnot\nec\phi\psi$:

  We reduce $\pf i {\nec\phi\psi} \in \Gamma$. The resulting branch is
  $\Gamma \cup \{\pf j \psi\}$ for some $j$ such that $\pr{i}{j}{\phi}
  \in \Gamma$. Since $M, f$ satisfies $\Gamma$, $M, f(i) \sat
  \nec\phi\psi$ and $\R{\phi}{f(i)}{f(j)}$. Hence, $M, f(j) \sat \psi$,
  i.e., $M, f$ satisfies $\pf{j}{\psi}$.
      
  We reduce $\pf i {\lnot\nec\phi\psi} \in \Gamma$. The resulting branch is
  $\Gamma \cup \{\pr{i}{j}{\phi}\} \cup \{\pf j {\lnot\psi}\}$ for some $j$ not
  occurring in~$\Gamma$. Since $M, f$ satisfies $\Gamma$, $M, f(i)
  \nsat \nec\phi\psi$, there is some $y \in U$ with $\R{\phi}{f(i)}{y}$
  such that $M, y \nsat \psi$. Extend $f$ to $f'$ by $f(j) = y$.  Then
  $M, f'$ also satisfies $\Gamma$. $M, f'(j) \nsat \psi$, i.e., $M, f$
  satisfies $\pf j {\lnot\psi}$. $\R{\phi}{f'(i)}{f'(j)}$ by definition
  of~$f'$, so $M, f'$ satisfies $\pr{i}{j}{\phi}$.

  The cut rule is sound: Suppose a branch~$\Gamma$ is extended by
  applying cut. We obtain two new branches, $\Gamma \cup \{\pf i
  \phi\}$ and $\Gamma \cup \{\pf i {\lnot\phi}\}$, where $i$ already
  occurs in~$\Gamma$. Since $M, f$ satisfies~$\Gamma$, $f(i)$ is
  defined. Either $M, f(i) \sat \phi$ or $M, f(i) \nsat \phi$. In the
  first case, $M, f$ satisfies $\Gamma \cup \{\pf i \phi\}$; in the
  second, it satisfies $\Gamma \cup \{\pf i {\lnot\phi}\}$.

  Rule ea is sound. Suppose $\Gamma$ contains $\pr i j \phi$. Since
  $M, f$ satisfies $\Gamma$, $\R\phi{f(i)}{f(j)}$. If $\Prop\phi =
  \Prop\psi$, then $R_\psi = R_\phi$ and we have $\R\psi{f(i)}{f(j)}$,
  i.e., $M, f$ satisfies~$\pr i j \psi$. Otherwise, there is some $y
  \in U$ where $M, y \nsat \phi$ but $M, y \sat \psi$ or $M, y \sat
  \phi$ but $M, y \nsat \psi$. Extend $f$ to $f'$ with $f'(k) = y$
  (since $k$ does not occur in $\Gamma$, this is possible). Then
  either $M, f'$ satisfies $\Gamma \cup \{\pf k {\lnot\phi}, \pf k
  \psi\}$ or it satisfies $\Gamma \cup \{\pf k {\phi}, \pf k
        {\lnot\psi}\}$.
  
  Now consider the rules of \hyperref[rules-VC]{Table~\ref*{rules-VC}}
  and assume that $M$ satisfies the respective condition of
  \hyperref[axioms-conditions]{Table~\ref*{axioms-conditions}}.
  \begin{enumerate}
    \item[R1.] $\Gamma$ contains $\pr{i}{j}{\phi}$ and the resulting branch
      contains $\pf{j}{\phi}$. By condition (1), $\Rf{\phi}{f(i)} \subseteq
      \Prop\phi$.  Since $M, f$ satisfies $\pr{i}{j}{\phi}$ we have
      $\R{\phi}{f(i)}{f(j)}$, i.e., $f(j) \in \Rf{\phi}{f(i)}$. Thus, $M,
      f(j) \sat \phi$.
    \item[R2.] $\Gamma$ contains $\pr{i}{j}{\phi}$ and $\pf j \psi$ and the
      extended branch contains $\pr{i}{k}{\psi}$, where $k$ does not occur
      in~$\Gamma$. Since $M, f$ satisfies $\pr{i}{j}{\phi}$ we have
      $\R{\phi}{f(i)}{f(j)}$, i.e., $f(j) \in \Rf{\phi}{f(i)}$. Since $M,
      f$ satisfies $\pf j \psi$ we have $f(j) \in \Prop\psi$, so $f(j)
      \in \Rf{\phi}{f(i)} \cap \Prop\psi$. By condition~(2),
      $\Rf{\psi}{f(i)} \neq \emptyset$, i.e., there is some $y \in U$
      such that $\R{\psi}{f(i)}{y}$. Extend $f$ to $f'$ with $f'(k) =
      y$. $M, f'$ satisfies $\pr{i}{k}{\psi}$.
    \item[R3.] $\Gamma$ contains $\pr i j \top$; $\pf{i}{\phi}$; and $\pf j
      {\lnot\phi}$. The resulting branch also contains $\pf j
      \phi$ and is thus closed. So in this case we have to show
      that $\Gamma$ is not satisfiable. If $M, f$ satisfies $\pr i j
      \top$ we have $\R{\top}{f(i)}{f(j)}$, i.e., $f(j) \in
      \Rf{\top}{f(i)}$. By condition (3), $f(i) = f(j)$. But this is
      impossible since $M, f(i) \sat \phi$ and $M, f(j) \sat
      \lnot\phi$.
    \item[R4.] In this case, $\Gamma$ is extended by adding $\pr i i \top$.
      By condition (4), $f(i) \in \Rf{U}{f(i)}$, and since $\Prop\top =
      U$ we have $\R{\top}{f(i)}{f(i)}$.
    \item[R5.] $\Gamma$ contains $\pr{i}{j}{\phi}$ and $\pf j \psi$ and is
      extended by $\pr{i}{j}{\phi \land \psi}$. Since $f(j) \in
      \Rf{\phi}{f(i)}$ and $f(j) \in \Prop\psi$, by (5) we get $f(j)
      \in \Rf{\phi \land \psi}{f(i)}$, i.e., $\R{\phi \land
        \psi}{f(i)}{f(j)}$.
    \item[R6.] $\Gamma$ contains $\pr i j \phi$; $\pf j \psi$; and
      $\pr i k {\phi \land \psi}$, and is extended by
      $\pr{i}{k}{\phi}$ and $\pf k \psi$.  Since $M, f$ satisfies
      $\Gamma$, we have $f(j) \in \Rf{\phi}{f(i)}$ and $f(j) \in
      \Prop\psi$. So $\Rf{\phi}{f(i)} \cap \Prop\psi \neq
      \emptyset$. By condition (6), $\Rf{\phi \land \psi}{f(i)}
      \subseteq \Rf{\phi}{f(i)} \cap \Prop\psi$. In other words, for
      any $y$ such that $\R {\phi \land \psi} {f(i)} y$ we have both
      $\R{\phi}{f(i)}{y}$ and $y \in \Prop\psi$. Since $M, f$
      satisfies $\pr i k {\phi \land \psi} \in \Gamma$, we have that
      $\R {\phi \land \psi} {f(i)} {f(k)}$, i.e., $f(k)$ is such
      a~$y$. Thus, $\R{\phi}{f(i)}{f(k)}$ and $M, f(k) \sat \psi$.
  \end{enumerate}
\end{proof}

\begin{cor}\label{soundness-sub-VC}
  Tableaux for \Ck{} and \CK{} are sound for Segerberg
  models. Tableaux for \Vc{} are sound for \VC-models
\end{cor}

\begin{proof}
  The conditions for \VC-models are only used in the verification of
  soundness of rules R1--6. Since \VC{} is sound for \VC-models, and
  \Vc{} has fewer rules than \VC, \Vc{} is also sound for \VC-models.
\end{proof}

\begin{defn}
  We say that $\Gamma$ \emph{\VC-entails} $\phi$, $\Gamma \sat_\VC$,
  iff for every \VC-model~$M$ and world~$x$ such that $M, x \sat \psi$
  for every $\psi \in \Gamma$, $M, x \sat \phi$.
\end{defn}

\begin{cor}
  If $\Gamma \vdash_\VC \phi$ then $\Gamma \sat_\VC \phi$.
\end{cor}

\begin{prop}
  Tableaux for \VC{} are complete, i.e., if $\Gamma \sat_\VC \phi$
  then there is a closed \VC-tableau for $\Gamma \vdash_\VC \phi$.
\end{prop}

\begin{proof}
  \cite{Segerberg1989} stated strong completeness of RCEC, RCEA,
  CM, CC, CN plus axioms S1--6 of
  \hyperref[axioms-conditions]{Table~\ref*{axioms-conditions}}.  By
  \hyperref[complete-CK]{Theorem~\ref*{complete-CK}}, it suffices to
  show that axioms S1--6 have closed \VC-tableaux. These can be found
  in the \hyperref[appendix]{Appendix}.
\end{proof}

Finally, a remark about the system \textbf{C2} of
\citet{Stalnaker1970}. Also known as \textbf{VCS}, it is the logic
characterized by Lewis as \VC{} plus conditional excluded middle,
$(\phi \cif \psi) \lor (\phi \cif \lnot\psi)$, or, in Chellas's
notation: $\nec\phi\psi \lor \nec\phi{\lnot\psi}$. Segerberg frames
for it are characterized by the condition that whenever $\R S x y$
and $\R S x z$, then $y = z$. A tableau rule for it would be
\[
\AxiomC{\pr i j \phi }
\noLine
\UnaryInfC{\pr i k \phi}
\noLine
\UnaryInfC{\pf j \psi}
\RightLabel{cem}
\UnaryInfC{\pf k \psi}
\DisplayProof
\]
The rule is clearly sound. Its addition results in a system complete
for \textbf{VCS}, as it can prove conditional excluded middle:
\begin{center}
  \begin{tableau}{to prove={\vdash \nec\phi\psi \lor \nec\phi{\lnot\psi}}}
      [\pf 1 {\lnot(\nec\phi\psi \lor \nec\phi{\lnot\psi})}, just=Ass
        [\pf 1 {\lnot\nec\phi\psi}, just=$\lnot\lor$:!u
          [\pf 1 {\lnot\nec\phi{\lnot\psi}}, just=$\lnot\lor$:!uu
            [\pf 2 \psi, just=$\lnot\Box$:!uu
              [\pr 1 2 \phi, just=$\lnot\Box$:!uuu
                [\pf 3 {\lnot\psi}, just=$\lnot\Box$:!uuu
                  [\pr 1 3 \phi, just=$\lnot\Box$:!uuuu
                    [\pf 3 \psi, just=cem:{!uuu,!u,!uuuu},close={:!uu,!c}
                    ]
                  ]
                ]
              ]
            ]
          ]
        ]
      ]
  \end{tableau}
\end{center}

\section{Challenges for Proofs of Cut-Free Completeness}

We have shown that \CK{} and \VC{} are sound and complete with respect
to Segerberg models and \VC-models, respectively. The completeness
theorem relies on the presence of the cut rule. What are the prospects
of proving completeness without the cut rule? First, let us review the
proof of cut-free completeness for~\Ck{} due to \citet{Priest2008}.

\begin{defn}
  A \emph{Priest model} $M = \langle U, R, V\rangle$ consists of a set of
  worlds~$U \neq \emptyset$, an accessibility relation $R\colon
  \mathrm{Frm} \to \wp(U \times U)$ indexed by \emph{formulas}, and a
  variable assignment $V\colon \mathit{Var} \to \wp{U}$.

  We write $R_\phi$ for $R(\phi)$ and $\Rf \phi x$ for $\{ y \mid \R \phi x y\}$.
\end{defn}

\begin{defn}
  Truth of a formula $\phi$ at a world~$x$ in $M$, $M, x \sat \phi$, 
  is defined exactly as for Segerberg models.
\end{defn}

\begin{thm}[\citealt{Priest2008}, \S5.9]
  \Ck-tableaux are sound for Priest models.
\end{thm}

\begin{proof}
  As in \hyperref[soundness-VC]{Theorem~\ref*{soundness-VC}}.
\end{proof}

\begin{thm}[\citealt{Priest2008}, \S5.9]
  \Ck-tableaux are complete for Priest models, i.e., if $\Delta$ has no
  closed \Ck-tableau, then it has a Priest model.
\end{thm}

\begin{proof}
  Call a branch \emph{complete} if for every rule that can be applied
  on the branch has been applied. Since any finite branch contains
  only finitely many prefixed formulas, it can be extended by applying
  all finitely many rules that are applicable. So if $\Delta$ has no
  closed tableau, there is a tableau with $\Delta$
  as assumptions which contains at least one complete open
  branch~$\Gamma$.

  Let $U$ be the set of indices occurring on~$\Gamma$. Let
  \[
  V(p) = \{ i \mid \pf i p \in \Gamma \}.
  \]
  Set $\R \phi i j$ iff $\pr{i}{j}{\phi} \in \Gamma$.

  We show that if $\pf{i}{\theta} \in \Gamma$ then $M, i \sat \theta$ and if
  $\pf i {\lnot\theta} \in \Gamma$ then $M, i \nsat \theta$ by induction
  on~$\theta$.

  \begin{enumerate}
  \item $\theta$ is atomic: If $\pf i p \in \Gamma$, $M, i \sat p$ by
    definition of~$V$. If $\pf i {\lnot p} \in \Gamma$, then since
    $\Gamma$ is open, $\pf i p \notin \Gamma$. Thus $i \notin V(p)$, and
    $M, i \nsat p$.
  \item $\theta \equiv \lnot \phi$: Suppose $\pf i {\lnot \phi} \in
    \Gamma$. If $\phi$ is of the form $\lnot \psi$, since $\Gamma$ is
    complete, $\pf i \psi \in \Gamma$. By induction hypothesis, $M, i
    \sat \psi$ and so $M, i \sat \lnot\lnot\psi$, i.e., $M, i \sat
    \lnot\phi$. If $\phi$ is not of the form $\lnot\psi$, the case
    $\pf i {\lnot\phi} \in \Gamma$ will be treated in one of the cases
    below. If $\pf i {\lnot\lnot\phi} \in \Gamma$, then again $\pf i
    \phi \in \Gamma$ and we have $M, i \nsat \lnot \phi$.
  \item $\theta \equiv \phi \land \psi$: Suppose $\pf i {\phi \land \psi}
    \in \Gamma$. Since $\Gamma$ is closed, $\pf{i}{\phi} \in \Gamma$ and
    $\pf i \psi \in \Gamma$. By induction hypothesis, $M, i \sat \phi$
    and $M, i \sat \psi$, so $M, i \sat \phi \land \psi$. If $\pf i
    {\lnot(\phi \land \psi)} \in \Gamma$, then since $\Gamma$ is
    closed, either $\pf i {\lnot\phi} \in \Gamma$ or $\pf i {\lnot\psi} \in
    \Gamma$. Thus, either $M, i \nsat \phi$ or $M, i \nsat \psi$ by
    IH, and $M, i \nsat \phi \land \psi$.
  \item The cases for $\theta \equiv \phi \lor \psi$ and $\theta
    \equiv \phi \lif \psi$ are handled similarly.
  \item $\theta \equiv \nec\phi\psi$: Suppose $\pf i {\nec\phi\psi}
    \in \Gamma$. For every $j$ such that $\R{\phi}{i}{j}$,
    $\pr{i}{j}{\phi} \in \Gamma$ by definition of $R_\phi$. So for
    every $j$ such that $\R{\phi}{i}{j}$, $j$ occurs
    in~$\Gamma$. Since $\Gamma$ is complete, the $\Box$ rule must have
    been applied on~$\Gamma$ with index~$j$, i.e., $\pf j \psi \in
    \Gamma$.  By induction hypothesis, $M, j \sat \psi$. Hence, $M, j
    \sat \nec\phi\psi$.

    Now suppose $\pf i {\lnot\nec\phi\psi} \in \Gamma$. Since $\Gamma$ is
    complete, $\pf j {\lnot\psi} \in \Gamma$ and $\R{\phi}{i}{j} \in \Gamma$
    for some~$j$. By IH, $M, j \nsat \psi$. By definition of $R_\phi$,
    $\R{\phi}{i}{j}$. So, $M, i \nsat \nec\phi\psi$.
  \end{enumerate}
\end{proof}

Extending Priest's proof of cut-free completeness to \Vc{} would
involve a definition of a class of Priest models for which
\Vc-tableaux are sound and cut-free complete. The obvious approach
would be to reformulate Segerberg's conditions for Priest
models. Let's call a Priest model that satisfies the conditions of
\hyperref[Vc-conditions]{Table~\ref*{Vc-conditions}} a
\emph{\Vc-model}.

\begin{table}
  \begin{center}
  \begin{tabular}{ll}
    \hline\hline
    (1) & $\Rf \phi x \subseteq \Prop\phi$\\
    (2) & $\Rf \phi x \cap \Prop\psi \neq \emptyset \Rightarrow \Rf \psi x \neq \emptyset$\\
    (3) &
    $\Rf \top x \subseteq \{x\}$ \\
    (4) & $x \in \Rf \phi x$\\
    (5) &
    $\Rf \phi x \cap \Prop\psi \subseteq \Rf {\Prop{\phi \land \psi}}{x}$\\
    (6) & 
    $\Rf \phi x \cap \Prop\psi \neq \emptyset \Rightarrow 
    \Rf{\phi \land \psi}{x} \subseteq \Rf \phi x \cap \Prop\psi$
    \\\hline
\end{tabular}\end{center}
    \caption{Conditions on \Vc-models}
    \label{Vc-conditions}
\end{table}

We would now have to show that the model constructed from a tableau
branch which is also closed under rules R1--6 satisfies the
corresponding property in
\hyperref[Vc-conditions]{Table~\ref*{Vc-conditions}}.

Condition R1 poses no problem: We have to show that $\Rf \phi i
\subseteq \Prop\phi$ for all~$\phi$. Suppose $\R \phi i j$. By
definition of~$R$, $\pr i j \phi \in \Gamma$. Since $\Gamma$ is
complete, rule R1 has been applied on it, so $\pf j \phi \in
\Gamma$. By the result above, $M, j \sat \phi$, i.e., $j \in
\Prop\phi$.  

The same approach does not work for R2 ($\Rf \phi i \cap \Prop\psi
\neq \emptyset \Rightarrow \Rf \psi i \neq \emptyset$). For suppose $j
\in \Rf \phi i \cap \Prop\psi$, i.e., $\R \phi i j$ and $M, j \sat
\psi$.  By definition of $R$, $\pr i j \phi \in \Gamma$. However, $M,
j \sat \psi$ does not guarantee that $\pf j \psi \in \Gamma$. In fact,
if $\psi$ is not a subformula of $\Delta$, $\pf j \psi$ is guaranteed
\emph{not} to be in~$\Gamma$. 

For (3), suppose $\R \top i j$. Then $\pr i j \top \in
\Gamma$. Completeness under R3 only rules out $i \neq j$ if $\pf i
\psi \in \Gamma$ for some~$\psi$, which is not guaranteed. Etc.

To extend the proof of cut-free completeness to \VC{} with respect to
\VC-models, we face an even more difficult obstacle. For \VC{} models
have their accessibility relation indexed by propositions, not
formulas. So the definition of a Segerberg model~$M$ from an open
complete branch~$\Gamma$ would have to define $R_S$ on the basis of
which $\pr i j \phi$ are in~$\Gamma$. We could do this only if we
already knew which propositions $\phi$ expresses in~$M$---then we could
say that if $S = \Prop\phi$, $\R S i j$ iff $\pr i j \phi \in
\Gamma$. But of course we can't do this, since $M$---which determines
$\Prop\phi$---is not yet defined! An additional obstacle is that rule ea
does not guarantee that $\pr i j \phi$ iff $\pr i j \psi$ even when,
for all $k$, $\pf k \phi \in \Gamma$ iff $\pf k \psi \in \Gamma$.

\section{Conclusion}

The above considerations show that using the approach pioneered by
Chellas and Segerberg, and recently extended significantly by
Unterhuber and Schurz, point to a way of constructing analytic rules
for conditional logics (if not analytic tableaux calculi).  The rules
given are relatively straightforward translations of conditions on the
indexed accessibility relation into tableaux rules. A condition
expressible by a universal formula, such as S1 ($\forall i\forall j(\R
\phi i j \lif j \in \Prop \phi)$), translates into 
extension rules without conditions. Those involving existential quantifiers, such as S2, i.e.,
\[
\forall i\forall j((\R \phi i j \land j \in \Prop\psi) \lif \exists
k\, \R \psi i k),
\]
involve the introduction of prefixes new to the
branch.  It is plausible that a fragment of first-order logic can be
identified such that any condition expressible in that fragment can be
translated into a sound tableau rule. Identity (and uniqueness), e.g.,
the condition for \textbf{C2} ($\forall i\forall j\forall k((\R \phi i
j \land \R \phi i k) \lif j = k)$) is a bit trickier: here the
resulting rule does not force $j = k$ but only that all formulas
evaluate the same at $j$ and~$k$, i.e., $j$ and $k$ are indiscernible.
So, other systems can be dealt with in a similar manner, as long as
they are characterized by accessibility relations expressible in the
right way. The lack of a proof of cut-free completeness is of course
unsatisfying. Having candidate rules as well as a tentative analysis
of why available methods of establishing cut-free completeness fail
perhaps points in a direction of solving this open problem. If nothing
else, we have highlighted that Chellas's and Segerberg's approach to a
semantics for conditionals has not yet been sufficiently exploited in
the search for analytic proof systems for a class of logics
that includes systems as important as Lewis's~\VC.

\bigskip\noindent\textbf{Acknowledgements.}
  I'd like to thank the reviewer for the \emph{AJL} for their comments.

\bibliography{cstab}

\newpage
\section*{Appendix}
\label{appendix}

\subsection*{Tableaux for \Ck-axioms}

\begin{enumerate}
\item\label{CM} \Ck-tableau for CM:
  
  \begin{tableau}{to prove=
      {\nec\phi{(\psi \land \theta)} \vdash_\Ck
        \nec\phi\psi \land \nec\phi\theta}}
    [\pf 1 {\nec\phi{(\psi \land \theta)}}, just = Ass
      [\pf 1 {\lnot(\nec\phi\psi \land \nec\phi\theta)}, just = Ass
        [\pf 1 {\lnot\nec\phi\psi}, just = $\lnot\land$:!u
          [\pf 2 {\lnot\psi}, just = $\lnot\Box$:!u
            [\pr 1 2 \phi, just = $\lnot\Box$:!uu
              [\pf 2 {\psi \land \theta}, just = $\Box$:{!uuuuu,!u}
                [\pf 2 \psi, just = $\land$:!u
                  [\pf 2 \theta, just = $\land$:!uu, close={:!u,!uuuu}]
                ]
              ]
            ]
          ]
        ]
        [\pf 1 {\lnot\nec\phi\theta}, just = $\lnot\land$:!u
          [\pf 2 {\lnot\theta}, just = $\lnot\Box$:!u
            [\pr 1 2 \phi, just = $\lnot\Box$:!uu
              [\pf 2 {\psi \land \theta}, just = $\Box$:{!uuuuu,!u}
                [\pf 2 \psi, just = $\land$:!u
                  [\pf 2 \theta, just = $\land$:!uu, close={:!c,!uuuu}]
                ]
              ]
            ]
          ]
        ]
      ]
    ]
  \end{tableau}
\item \Ck-tableau for CC:
  
  \begin{tableau}{to prove=
      {\nec\phi\psi \land \nec\phi\theta \vdash_\Ck
        \nec\phi{(\psi \land \theta)}}}
    [\pf 1 {\nec\phi\psi \land \nec\phi\theta}, just = Ass
      [\pf 1 {\lnot\nec\phi{(\psi \land \theta)}}, just = Ass
        [\pf 1 {\nec\phi\psi}, just = $\land$:!uu
          [\pf 1 {\nec\phi\theta}, just = $\land$:!uuu
            [\pr 1 2 \phi, just = $\lnot\Box$:!uuu
              [\pf 2 {\lnot(\psi \land \theta)}, just = $\lnot\Box$:!uuuu
                [\pf 2 {\lnot\psi}, just = $\lnot\land$:!u
                  [\pf 2 \psi, just = $\Box$:{!uuuuu,!uuu}, close={:!u,!c}]
                ]
                [\pf 2 {\lnot\theta}, just = $\lnot\land$:!u
                  [\pf 2 \theta, just = $\Box$:{!uuuu,!uuu}, close={:!u,!c}]
                ]
              ]
            ]
          ]
        ]
      ]
    ]
  \end{tableau}\hfill
  \item\label{CN} \Ck-tableau for CN:
  
  \begin{tableau}{to prove={\vdash \nec\phi\top}}
    [\pf 1 {\lnot\nec\phi\top}, just = Ass
      [\pr 1 2 \phi, just = $\lnot\Box$:!u
        [\pf 2 \lnot\top, just = $\lnot\Box$:!uu, close={:!c}]
      ]
    ]
  \end{tableau}
\end{enumerate}

\subsection*{\Vc-tableaux for axioms S1--6}

\begin{enumerate}
\item
  \begin{tableau}{to prove={\vdash \nec\phi\phi}}
    [\pf 1 {\lnot\nec\phi\phi}, just = Ass
      [\pr 1 2 \phi, just = $\lnot\Box$:!u
        [\pf 2 {\lnot\phi}, just = $\lnot\Box$:!uu
          [\pf 2 \phi, just = R1:!uu, close={:!u,!c}
          ]
        ]
      ]
    ]
  \end{tableau}
\item 
  \begin{tableau}{to prove={\poss\phi\psi \vdash \poss\psi\top}}
    [\pf 1 {\poss\phi\psi}, just = Ass
      [\pf 1 {\lnot\poss\psi\top}, just = Ass
        [\pr 1 2 \phi, just = $\Diamond$:!uu
          [\pf 2 \psi, just = $\Diamond$:!uuu
            [\pr 1 2 \psi, just = R2:{!uu,!u}
              [\pf 2 \lnot\top, just = $\lnot\Diamond$:{!uuuu,!u}, close={:!c}]
            ]
          ]
        ]
      ]
    ]
  \end{tableau}
\item 
  \begin{tableau}{to prove={\phi \vdash \nec\top\phi}}
    [\pf 1 \phi, just = Ass
      [\pf 1 {\lnot\nec\top\phi}, just = Ass
        [\pr 1 2 \top, just=$\lnot\Box$:!u,
          [\pf 2 \lnot\phi, just = $\lnot\Box$:!uu,
            [\pf 2 \phi, just = R3:{!uuuu,!uu,!u}, close={:!u,!c}
            ]
          ]
        ]
      ]
    ]
  \end{tableau}
\item 
  \begin{tableau}{to prove={\phi \vdash \poss\top\phi}}
    [\pf 1 \phi, just = Ass
      [\pf 1 {\lnot\poss\top\phi}, just = Ass
        [\pr 1 1 \top, just=R4
          [\pf 1 {\lnot\phi}, just=$\lnot\Diamond$:{!uu,!u}, close={:!uuu,!c}
          ]
        ]
      ]
    ]
  \end{tableau}
\item 
  \begin{tableau}{to prove=
      {\nec{(\phi \land \psi)}\theta \vdash \nec\phi{(\psi \lif \theta)}}}
    [\pf 1 {\nec{(\phi \land \psi)}\theta}, just = Ass
      [\pf 1 {\lnot\nec\phi{(\psi \lif \theta)}}, just = Ass
        [\pr 1 2 \phi, just = $\lnot\Box$:!u
          [\pf 2 {\lnot(\psi \lif \theta)}, just = $\lnot\Box$:!uu
            [\pf 2 \psi, just=$\lnot\lif$:!u
              [\pf 2 \lnot\theta, just = $\lnot\lif$:!uu
                [\pr 1 2 {\phi \land \psi}, just = R5:{!uuuu,!uu}
                  [\pf 2 \theta, just = $\Box$:{!uuuuuuu,!u}, close={:!uu,!c}
                  ]
                ]
              ]
            ]
          ]
        ]
      ]
    ]
  \end{tableau}
\item 
  \begin{tableau}{to prove=
      {\poss\phi\psi, \nec\phi{(\psi \lif \theta)} \vdash \nec{\phi
          \land \psi}\theta}}
    [\pf 1 {\poss\phi\psi}, just = Ass
      [\pf 1 {\nec\phi{(\psi \lif \theta)}}, just = Ass
        [\pf 1 {\nec{\phi \land \psi}\theta}, just = Ass
          [\pr 1 2 \phi, just = $\Diamond$:!uuu
            [\pf 2 \psi, just = $\Diamond$:!uuuu
              [\pr 1 3 {(\phi \land \psi)}, just = $\lnot\Box$:!uuu
                [\pf 3 \lnot\theta, just = $\lnot\Box$:!uuuu
                  [\pr 1 3 \phi, just=R6:{!uuuu,!uuu,!uu}
                    [\pf 3 \psi, just=R6:{!uuuuu,!uuuu,!uuu}
                      [\pf 3 {\psi \lif \theta}, just = $\Box$:{!uuuuuuuu,!uu}
                        [\pf 3 {\lnot\psi}, just = $\lif$:!u,close={:!uu,!c}]
                        [\pf 3 {\theta}, just = $\lif$:!u,close={:!uuuu,!c}]
                      ]
                    ]
                  ]
                ]
              ]
            ]
          ]
        ]
      ]
    ]
  \end{tableau}
  
\end{enumerate}

\end{document}